\documentclass[3p,8pt]{elsarticle}

\usepackage{graphicx}
\usepackage{amsmath,amssymb}
\usepackage{amsthm}
\usepackage{arydshln}
\usepackage{rotate,rotating}
\usepackage{epsf,epic,eepic}
\usepackage{bmpsize}
\usepackage{multirow} 
\usepackage{bm}
\usepackage{mathrsfs}
\usepackage{color}

\newcommand{\mc}[1]{\mathcal{#1}}

\newtheorem{theorem}{Theorem}

\usepackage{amssymb}

\journal{ArXiv}
\begin{document}

\begin{frontmatter}

\title{Global stability properties of a class of renewal epidemic models with variable susceptibility}

 \author[AITHM]{Michael T. Meehan}
\author[Daniel]{Daniel G. Cocks}
\author[AITHM]{Emma S. McBryde}

\address[AITHM]{Australian Institute of Tropical Health and Medicine, James Cook University, Townsville, Australia}
\address[Daniel]{Research School of Science and Engineering, Australian National University, Canberra, Australia}

\date{Received: date / Accepted: date}

\begin{abstract}

We investigate the global dynamics of a renewal-type epidemic model with variable susceptibility. We show that in this extended model there exists a unique endemic equilibrium and prove that it is globally asymptotically stable when $R_0 > 1$, i.e. when it exists. We also show that the infection-free equilibrium, which exists always, is globally asymptotically stable for $R_0 \leq 1$.



\end{abstract}


\end{frontmatter}






\section{Introduction}
\label{sec:Introduction}

In a recent article,~\citep{meehan2019global} investigated the asymptotic dynamics of a general class of renewal epidemic models for which both the force of infection and infected removal rates are arbitrary functions of an individual's infection age~\citep{kermack1927contribution,DIEKMANN1977459, metz1986,breda2012formulation}. By identifying appropriate Lyapunov functionals of the form $g(x) = x - 1 - \log x$~\citep{Korobeinikov2002955}, the authors were able to establish that the infection-free and endemic system equilibria were globally asymptotically stable when the basic reproduction number $R_0 \leq 1$ and $ > 1$, respectively (see also~\citep{magal2010lyapunov}). Here we extend this investigation by considering a class of renewal epidemic models that account for variable susceptibility to infection among the susceptible cohort.

The remainder of the paper is constructed as follow: in the next section we define the variable susceptibility model and introduce the relevant model parameters. We also discuss the infinite-dimensional phase-space of our system and introduce several important definitions. Following this, in section~\ref{sec:eq_points}, we derive expressions for the equilibrium solutions of the model system and determine the necessary and sufficient conditions for their existence. Then in section~\ref{sec:stability} we use the direct Lyapunov method to prove that the infection-free and endemic equilibria are globally asymptotically stable for $R_0\leq 1$ and $> 1$, respectively.

Both this work and the previous analysis~\citep{meehan2019global} follow a long list of studies that have successfully invoked the direct Lyapunov method to establish the global stability properties of dynamical system equilibria. For relevant references within the domain of epidemic modelling see, for instance,~\citep{li1995global,Li1999191,Korobeinikov2002955,Korobeinikov2004,Korobeinikov2008,McCluskey2008518,McCluskey2009603,McCluskey201055,MCCLUSKEY20103106,magal2010lyapunov,Huang2011,Bichara2013,MARTCHEVA2013225,CHEN201616}.

\section{Model description}
\label{sec:model}


The general Kermack-McKendrick epidemic model~\citep{kermack1927contribution}, describing the spread of an infection through a susceptible population, can be written in terms of the following set of equations: 
\begin{align}
\frac{dS(t)}{dt} &= \lambda - \mu S(t) - F(t)S(t),\nonumber\\
F(t) &= \int_0^{\bar\tau} A(\tau) F(t-\tau) S(t-\tau)\,d\tau\label{eq:FKM}
\end{align}
where $S(t)$ is the fraction of the population that is susceptible to infection at time $t$, and $F(t)$ is the force of infection. The quantity $A(\tau)$ appearing in the renewal definition of the force of infection~\eqref{eq:FKM} is the infectivity kernel and gives the expected contribution to $F(t)$ for individuals who have been infected for $\tau$ units of time. The parameter $\bar\tau$ is the maximum infection-age at which an individual remains infectious:
\begin{equation}
\bar\tau = \sup\left\{\tau \geq 0 \: : \: A(\tau) > 0\right\}.\label{eq:tau_bar}
\end{equation}
We have also included the demographic parameters $\lambda$ and $\mu$, which give the constant birth/recruitment rate and per-capita death rate of susceptible individuals respectively, to generate a model that allows for endemic behaviour.
In this model all (susceptible) individuals are assumed to be equally susceptible to infection.

In this article we would like to generalize the Kermack-McKendrick model~\eqref{eq:FKM} to account for varying levels of susceptibility among individuals in the S class. In particular we would like to decompose the larger susceptible population $S(t)$ into sub-populations $S(t,\sigma)$ to which we assign an additional label $\sigma$ to identify their level of susceptibility. 
More precisely, we introduce the measure $S(t,\sigma)$ to denote the density of susceptible individuals with susceptibility index $\sigma \in \Sigma$, where $\Sigma$ is a measurable space over a finite set\footnote{The extension to more general sets is relatively straightforward.}:
\begin{equation}
\mbox{Total susceptible population (at time $t$)} = \int_\Sigma S(t,\sigma) d\sigma <\infty.\nonumber
\end{equation}
For simplicity, we assume that the susceptibility index $\sigma$ is a static label that is assigned at birth and that it is not affected by aging or infection. We also assume that an individual's infectivity is not related to their susceptibility such that we still have $A \equiv A(\tau)$. With these assumptions we generalize the model equations given above to the system
\begin{align}
\frac{dS(t,\sigma)}{dt} &= \lambda(\sigma) - \mu S(t,\sigma) - \eta(\sigma)F(t)S(t,\sigma),\nonumber\\
F(t) &= \int_\Sigma\int_0^{\bar\tau} A(\tau)\eta(\sigma)F(t-\tau)S(t - \tau,\sigma)\,d\tau\,d\sigma. \label{eq:model}
\end{align}
Here we have introduced the functions $\eta(\sigma)$ and $\lambda(\sigma)$ which respectively describe the relative susceptibility and birth/recruitment of individuals with label $\sigma$. To ensure that the S compartment only contains individuals that are capable of becoming infected, that all susceptibility states are continually replenished, and finally that the evolution of the susceptibles $S(t,\sigma)$ is sufficiently smooth, we impose the constraints
\begin{equation}
\lambda(\sigma), \:\eta(\sigma) \in L^\infty_+(\Sigma).\nonumber
\end{equation}
Moreover, we assume that the maximum infection-age $\bar\tau$ is finite:
\begin{equation}
\bar\tau < \infty\nonumber
\end{equation}
otherwise, the lack of compactness in the infinite case makes the problem much more difficult~\citep{doi:10.1137/060659211, DIEKMANN2012819}.

To simplify the analysis that follows, we introduce the mean susceptibility at the infection-free and endemic equilibria, defined respectively as
\begin{equation}
\eta^0 = \int_\Sigma \eta(\sigma)S^0(\sigma)\,d\sigma\label{eq:eta0}
\end{equation}
and
\begin{equation}
\bar\eta = \int_\Sigma \eta(\sigma)\bar S(\sigma)\,d\sigma\label{eq:etabar}
\end{equation}
where $S^0(\sigma)$ and $\bar S(\sigma)$ are the $\sigma$-distributions of susceptibles at the infection-free and endemic equilibria respectively (see section~\ref{sec:eq_points}). 

To calculate the basic reproduction number for the model~\eqref{eq:model} we sum the expected contribution to the force of infection $A(\tau)$ over all infection ages and multiply this quantity by the mean susceptibility of the infection-free susceptible population:
\begin{align}
R_0 &= \int_\Sigma \eta(\sigma)S^0(\sigma)\,d\sigma \int_0^{\bar\tau} A(\tau)\,d\tau ,\nonumber\\
&= \eta^0  \int_0^{\bar\tau} A(\tau)\,d\tau.\label{eq:R0}
\end{align}
Note that for the special (i.e. homogenous) case $\eta(\sigma) = 1$ we recover the familiar expression $R_0 = S^0\int_0^{\bar\tau} A(\tau)\,d\tau$ where $S^0$ is the total number of susceptibles at the infection-free equilibrium~\citep{diekmann1990definition}.



As discussed previously~\citep{meehan2019global} we see that a full prescription of the model~\eqref{eq:model} requires a specification of the entire history of the susceptible population and the force of infection over time, from $t=-\bar\tau$ up to the present ($t=0$). Additionally, for the variable susceptibility case we must also provide the $\sigma$-distribution of susceptibles over the measure space $\Sigma$. In this regard the present state of the system $P = (\mc S, \mc F)$, is described by a set of functions $\mc S$ and $\mc F$, where $\mc F$ is defined over the interval $[-\bar\tau, 0]$ and $\mc S$ is a (continuous) function from the time interval $[-\bar\tau, 0]$ into $L^\infty_+(\Sigma)$. 
Hence, in order to ensure the necessary smoothness and compactness properties of our system trajectories (see below), we choose the initial conditions
\begin{equation}
\mc S_0 \in C^0_+([-\bar\tau, 0], L^\infty_+(\Sigma)) \qquad \mbox{and} \qquad \mc F_0\in L^1_+(-\bar\tau,0).
\end{equation}
In general, our phase-space is the infinite-dimensional product topology
\begin{equation}
\Omega = C^0_+([-\bar\tau, 0], L^\infty_+(\Sigma))\times L^1_+(-\bar\tau,0)
\end{equation}
which is a Banach space which we assume takes the natural norm. With this choice of state-space, standard arguments show that the model~\eqref{eq:model} is well defined. Additionally, the model equations~\eqref{eq:model} induce a continuous semiflow $\Phi_t : \Omega\rightarrow \Omega$ where the system trajectory is given by $(\mc S_t, \mc F_t) \in \Omega$ with
\begin{equation}
\mc S_t(s,\sigma) = S(t+s,\sigma),\qquad \mc F_t(s) = F(t+s),\qquad s\in [-\bar\tau, 0]. \nonumber
\end{equation}

Following the proof of Lemma 1 in~\citep{meehan2019global}, which invokes the smoothing properties of convolution integrals described in~\citep{Mikusinski}, it is straightforward to show that if the infectivity kernel $A$ is of bounded variation, i.e. $A\in BV_+([0,\bar\tau])$, we eventually have that $\mc S_t \in C^1_+$ and $\mc F_t \in AC_+$  (that is, $\mc F_t$ is absolutely continuous). Therefore, if we assume that $A\in BV_+$, system trajectories generated by $\Phi_t$ that originate in $\Omega$ are eventually bounded and relatively compact. In this case the $\omega$-limit set of ~\eqref{eq:model} is non-empty and we may employ the infinite-dimensional form of LaSalle's invariance principle~\citep[Theorem~5.17]{hsmith_textbook} to establish the global asymptotic stability of our model equilibria.

Finally, in the following, we decompose $\Omega$ into an ``interior'' and a ``boundary'' set $\widehat\Omega$ and $\partial\Omega$, respectively. Here, we do not refer to topological concepts, but rather to the interpretation in view of our application. For all initial values in the interior $\widehat\Omega$, the force of infection is non-zero. The elements in the boundary set $\partial\Omega$, in turn, do have a vanishing force of infection and therefore lead to trivial dynamics. To be more precise, we define
\begin{equation}
\widehat{\Omega} = \left\{(\mc S, \mc F) \in \Omega \: : \: \exists \, a\in [0,\bar\tau] \: \:\mbox{s.t.} \: \int_\Sigma\int_{0}^{\bar\tau} A(\tau + a) \mc F(-\tau) \mc S(-\tau, \sigma)\,d\tau d\sigma > 0 \right\}\nonumber
\end{equation}
and
\begin{equation}
\partial\Omega = \Omega \setminus \widehat{\Omega}.\nonumber
\end{equation}

\section{Equilibrium points}
\label{sec:eq_points}

We now aim to evaluate the equilibrium $\sigma$-distributions of the susceptible populations both in the absence, $S^0(\sigma)$, and presence, $\bar S(\sigma)$, of infection and determine conditions for their existence.

First, we note that at equilibrium, the system~\eqref{eq:model} becomes
\begin{equation}
0 = \lambda(\sigma) - \mu S^*(\sigma) - \eta(\sigma)F^* S^*(\sigma),\label{eq:sdotzero}
\end{equation}
and
\begin{equation}
F^* = F^*\int_\Sigma \eta(\sigma) S^*(\sigma)\,d\sigma  \int_0^{\bar\tau} A(\tau)\,d\tau 
.\label{eq:Fzero}
\end{equation}

From these equations it is straightforward to identify the infection-free susceptible distribution by setting the force of infection $F^* = 0$ in~\eqref{eq:sdotzero}:
\begin{equation}
S^0(\sigma) = \frac{\lambda(\sigma)}{\mu}.\label{eq:S0}
\end{equation}
Hence, the infection-free equilibrium, $P^0 = (\mc S^0, \mc F^0) = (\lambda(\sigma)/\mu, 0)$, always exists, and, given that $\lambda(\sigma)\in L^\infty_+$, we have that $\mc S^0 > 0$ for all $\sigma\in\Sigma$.
 
Next, we can find the endemic distribution $\bar S(\sigma)$ (for which $\bar F \neq 0$) by re-arranging equation~\eqref{eq:sdotzero} to give
\begin{equation}
\bar S(\sigma) = \frac{\lambda(\sigma)}{\mu + \eta(\sigma)\bar F}. \label{eq:Sbar} 
\end{equation}
Here we find that at the endemic equilibrium susceptible individuals are depleted from their infection-free distribution according to their relative susceptibility, $\eta(\sigma)$. 

It remains now to determine the conditions for the existence of $\bar S(\sigma)$, or equivalently, to determine the sign of the endemic force of infection $\bar F$. We start by noting the identity
\begin{equation}
1 = \int_\Sigma \eta(\sigma) \bar S(\sigma) d\sigma \int_0^{\bar\tau} A(\tau)\,d\tau = \bar\eta \int_0^{\bar\tau} A(\tau)\,d\tau\label{eq:identity}
\end{equation}
which follows from~\eqref{eq:Fzero} for the case $\bar F\neq0$.

If we then substitute our solution for $\bar S(\sigma)$ (eq.~\eqref{eq:Sbar}) into this expression we obtain
\begin{equation}
1 = \int_\Sigma \frac{\eta(\sigma)\lambda(\sigma)}{\mu + \eta(\sigma)\bar{F}}\,d\sigma\int_0^{\bar\tau} A(\tau)\,d\tau. \label{eq:identity_full}
\end{equation}
Here, given the definitions of $\eta(\sigma)$ and $A(\tau)$, and the restrictions placed on them (e.g. both are non-negative functions), we observe that the right-hand side of~\eqref{eq:identity_full} is a positive, strictly decreasing function of $\bar F$. Further, from~\eqref{eq:R0} and~\eqref{eq:S0} we see that the right-hand side of~\eqref{eq:identity_full} evaluated at $\bar F = 0$ becomes
\begin{equation}
\int_\Sigma \frac{\eta(\sigma)\lambda(\sigma)}{\mu}\,d\sigma\int_0^{\bar\tau} A(\tau)\,d\tau = R_0.\nonumber
\end{equation}
Together, from these properties we can deduce that positive solutions to~\eqref{eq:identity_full} exist if, and only if, $R_0 > 1$. In other words, the endemic equilibrium, $\bar{P} = (\bar{\mc {S}}, \bar{\mc {F}})\in\widehat\Omega$ if, and only if, $R_0 > 1$. For the boundary case $R_0 = 1$, we find that $\bar F = 0$ and the endemic and infection-free equilibria coincide, i.e. $\bar{P} = P^0\in\partial\Omega$.


Ultimately, our goal will be to establish that i) when $R_0 \leq 1$ all system trajectories of~\eqref{eq:model} within $\Omega$ asymptotically approach the infection-free equilibrium point $P^0 \in \partial\Omega$ and ii) when $R_0 > 1$ trajectories that originate in $\Omega$ asymptotically approach the endemic equilibrium $\bar{P} \in \widehat{\Omega}$, except those that originate in $\partial\Omega$ which approach $P^0$.

\section{Global stability analysis}
\label{sec:stability}


\subsection{Infection-free equilibrium}
\begin{theorem}
\label{the:ife}
The infection-free equilibrium point $P^0$ of the system~\eqref{eq:model} is globally asymptotically stable in $\Omega$ for $R_0 \leq 1$. However, if $R_0 > 1$, solutions of~\eqref{eq:model} starting sufficiently close to $P^0$ in $\Omega$ move away from $P^0$, except those starting within the boundary region $\partial \Omega$ which approach $P^0$.
\end{theorem}

\begin{proof}[Proof of Theorem~\ref{the:ife}]
To verify theorem~\ref{the:ife} we define the forward invariant set $D = \Phi_{\bar\tau}(\Omega)$. Importantly, any trajectory that originates in $\Omega$ enters $D$ either at, or before $t=\bar\tau$ and, from~\eqref{eq:model}, we have that $\mc S(0) > 0$ for all $(\mc S, \mc F)\in D$.

Now, consider the Lyapunov functional $U\, : \,  D\rightarrow \mathbb{R}_+$ defined by
\begin{equation}
U(\mc S, \mc F) = \int_\Sigma  S^0(\sigma)\, g\left(\frac{\mc S(0,\sigma)}{S^0(\sigma)}\right)\,d\sigma + \int_\Sigma\int_0^{\bar\tau}\xi(\tau) \eta(\sigma)\mc F(-\tau)\mc S(-\tau, \sigma)\,d\tau\,d\sigma\label{eq:U}
\end{equation}
where
\begin{align}
g(x) = x - 1 - \log x \qquad \mbox{and} \qquad \xi(\tau) = \eta^0\int_\tau^{\bar\tau} A(\rho)\,d\rho.\label{eq:getadef}
\end{align}
We note that the kernel $\xi$ has the following properties:
\begin{equation}
\xi(0) = R_0, \quad \xi(\bar\tau) = 0 \qquad \mbox{and} \qquad \frac{d\xi(\tau)}{d\tau} = - \eta^0A(\tau). \label{eq:xicons}
\end{equation}
Importantly, the functional $U(\mc S, \mc F) \geq 0$ and has a global minimum at the infection-free equilibrium $P^0$.

Evaluating the Lyapunov functional $U(\mc S, \mc F)$ along system trajectories $(\mc S_t, \mc F_t)$ we then have
\begin{align}
U(\mc S_t, \mc F_t) &= \int_\Sigma  S^0(\sigma)\, g\left(\frac{\mc S_t(0,\sigma)}{S^0(\sigma)}\right)\,d\sigma + \int_\Sigma\int_0^{\bar\tau}\xi(\tau) \eta(\sigma)\mc F_t(-\tau)\mc S_t(-\tau, \sigma)\,d\tau\,d\sigma,\nonumber\\
&= \int_\Sigma  S^0(\sigma)\, g\left(\frac{S(t,\sigma)}{S^0(\sigma)}\right)\,d\sigma + \int_\Sigma\int_0^{\bar\tau}\xi(\tau)\eta(\sigma) F(t-\tau) S(t-\tau, \sigma)\,d\tau\,d\sigma\nonumber
\end{align}
where in the second line we have re-introduced the notation $\mc S_t(s,\sigma) = S(t+s,\sigma)$ and $\mc F_t(s) = F(t+s)$. Next, in order to compute derivatives of $U$ we rewrite the integral in the second term such that
\begin{equation}
U(\mc S_t, \mc F_t) = \int_\Sigma  S^0(\sigma)\, g\left(\frac{S(t,\sigma)}{S^0(\sigma)}\right)\,d\sigma + \int_\Sigma\int_{t-\bar\tau}^{t}\xi(t - s)\eta(\sigma) F(s) S(s, \sigma)\,d\tau\,d\sigma.
\end{equation}

Differentiating the first term in our Lyapunov functional $U$ with respect to time gives:
\begin{align}
&\quad \frac{d}{dt} \int_\Sigma  S^0(\sigma)\, g\left(\frac{ S(t,\sigma)}{S^0(\sigma)}\right)\,d\sigma\nonumber\\
&= \int_\Sigma \left(1 - \frac{S^0(\sigma)}{S(t,\sigma)}\right)\,\frac{dS(t,\sigma)}{dt}\,d\sigma,\nonumber\\
&= \int_\Sigma \left(1 - \frac{S^0(\sigma)}{S(t,\sigma)}\right)\left(\lambda(\sigma) - \mu S(t,\sigma) - \eta(\sigma) F(t) S(t,\sigma)\right)\,d\sigma,\nonumber\\
&= \int_\Sigma \left(1 - \frac{S^0(\sigma)}{S(t,\sigma)}\right)\left(\lambda(\sigma) - \mu S(t,\sigma)\right)\,d\sigma  - F(t)\int_\Sigma \eta(\sigma)\left(S(t,\sigma) - S^0(\sigma)\right)\,d\sigma,\nonumber\\
&= -\mu \int_\Sigma S(t,\sigma)\left(1 - \frac{S^0(\sigma)}{S(t,\sigma)}\right)^2\,d\sigma - F(t)\int_\Sigma \eta(\sigma)S(t,\sigma)\,d\sigma  + \eta^0 F(t).\label{eq:der1}
\end{align}
Note that in the final line we have substituted in the identities $\lambda(\sigma) = \mu S^0(\sigma)$ and $\eta^0 = \int_\Sigma \eta(\sigma) S^0(\sigma)\,d\sigma$.
Next, we differentiate the second term in $U$ and use the properties of $\xi$ (see equation~\eqref{eq:xicons}) to get
\begin{align}
&\quad \frac{d}{dt}\int_\Sigma\int_{t-\bar\tau}^{t}\xi(t - s) \eta(\sigma) F(s) S(s, \sigma)\,d\tau\,d\sigma \nonumber\\
&= \int_\Sigma \bigg[\xi(0)\eta(\sigma)F(t)S(t,\sigma) - \xi(\bar\tau)\eta(\sigma)F(t-\bar\tau)S(t-\bar\tau,\sigma)\bigg.\nonumber\\
&\qquad \left. + \int_{t-\bar\tau}^t \frac{d\xi(t-s)}{dt}\eta(\sigma)F(s)S(s,\sigma)\,d\tau\right]\,d\sigma,\nonumber\\
&= R_0 F(t)\int_\Sigma \eta(\sigma)S(t,\sigma)d\sigma - \eta^0 \int_\Sigma\int_{t-\bar\tau}^t A(t-s) \eta(\sigma)F(s) S(s,\sigma)\,d\tau\,d\sigma,\nonumber\\
&= R_0 F(t)\int_\Sigma \eta(\sigma)S(t,\sigma) d\sigma - \eta^0 F(t).\label{eq:der2}
\end{align}

Finally, combining~\eqref{eq:der1} and~\eqref{eq:der2} yields
\begin{align}
\frac{dU(\mc S_t, \mc F_t)}{dt} &= -\mu \int_\Sigma S(t,\sigma)\left(1 - \frac{S^0(\sigma)}{S(t,\sigma)}\right)^2\,d\sigma \nonumber\\
&\qquad\quad  - (1 - R_0)F(t)\int_\Sigma \eta(\sigma)S(t,\sigma)\,d\sigma,\nonumber\\
&\leq 0.\label{eq:dU}
\end{align}
We emphasize that we know for a trajectory $(\mc S_t, \mc F_t)\in D \subset \Omega$, that for $t > \bar\tau$ we have $\mc F_t\in C^0([-\bar\tau, 0])$ such that~\eqref{eq:dU} is well defined and $U$ is a proper Lyapunov function on the domain $D$.


From~\eqref{eq:dU} we see that the derivative $\dot{U}(t) = 0$ if and only if $\mc S_t(0,\sigma) = S^0(\sigma)$ and either (a) $R_0 = 1$ or (b) $\mc F_t(0) = 0$. Therefore, the largest invariant subset in $\Omega$ for which $\dot{U} = 0$ is the singleton $\left\{P^0\right\}$. Given that the system orbit is eventually precompact, by the infinite-dimensional form of LaSalle's extension of Lyapunov's global asymptotic stability theorem~\citep[Theorem~5.17]{hsmith_textbook}, the infection-free equilibrium point $P^0$ is globally asymptotically stable in $\Omega$ for $R_0 \leq 1$.

Conversely, if $R_0 > 1$ and $\mc F_t(0) > 0$, the derivative $\dot U > 0$ if $S(t,\sigma)$ is sufficiently close to $S^0(\sigma)$. In this case, solutions starting sufficiently close to the infection-free equilibrium point $P^0$ leave a neighbourhood of $P^0$, except those starting in $\partial \Omega$. Since $\dot{U}\leq 0$ for solutions starting in $\Omega$, these solutions approach $P^0$ as $t\rightarrow\infty$.

\end{proof}

\subsection{Endemic equilibrium}

\begin{theorem}
\label{the:endemic}
If $R_0 > 1$ the endemic equilibrium point $\bar{P}$ is globally asymptotically stable in $\widehat{\Omega}$ (i.e. away from the boundary region $\partial \Omega$).
\end{theorem}

\begin{proof}[Proof of Theorem~\ref{the:endemic}]
Recall from the proof of theorem~\ref{the:ife} that when $R_0 > 1$, the force of infection $F(t)$ is bounded away from zero for $t > 0$. Therefore, when $R_0 > 1$, the interior region $\widehat\Omega$ is forward invariant, i.e. $\Phi_t \: : \: \widehat\Omega \rightarrow \widehat\Omega$. Hence, in analogy with theorem~\ref{the:ife}, for $R_0 > 1$ we may define the forward-invariant set $\widehat D = \Phi_{\bar\tau}(\widehat\Omega)$ where $\mc S,\mc F > 0$ for all $(\mc S, \mc F)\in \widehat D$.

In this case, we introduce the Lyapunov functional $ W \: : \: \widehat D \rightarrow \mathbb{R}_+$ defined as
\begin{equation}
W(\mc S, \mc F) = \int_\Sigma \bar S(\sigma)\, g\left(\frac{\mc S(0,\sigma)}{\bar S(\sigma)}\right)\,d\sigma + \int_\Sigma\int_0^{\bar\tau} \kappa(\tau) \bar v(\sigma)\,g\left(\frac{\mc F(-\tau)\mc S(-\tau, \sigma)}{\bar F \bar S(\sigma)}\right)\,d\tau\,d\sigma
\end{equation}
where $g(x)$ has been defined previously in~\eqref{eq:getadef}, 
\begin{equation}
\bar v(\sigma) = \eta(\sigma) \bar F \bar S(\sigma)
\end{equation}
and
\begin{equation}
\kappa(\tau) = \bar\eta \int_\tau^{\bar\tau} A(\rho)\,d\rho.
\end{equation}
Similar to before, the kernel $\kappa$ has the following properties:
\begin{equation}
\kappa(0) = 1, \quad \kappa(\bar\tau) = 0 \qquad \mbox{and} \qquad \frac{d\kappa(\tau)}{d\tau} = -\bar\eta A(\tau).\label{eq:kappacons}
\end{equation}

Following the same steps as in theorem~\ref{the:ife}, we evaluate the Lyapunov functional $W$ along system trajectories:
\begin{align}
W(\mc S_t, \mc F_t) 
&= \int_\Sigma \bar S(\sigma)\, g\left(\frac{ S(t,\sigma)}{\bar S(\sigma)}\right)\,d\sigma + \int_\Sigma\int_{t-\bar\tau}^{t} \kappa(t-s) \bar v(\sigma)\,g\left(\frac{F(s)S(s, \sigma)}{\bar F \bar S(\sigma)}\right)\,ds\,d\sigma.\nonumber
\end{align}

Next, we differentiate each term in $W(\mc S_t, \mc F_t)$ with respect to time to get
\begin{align}
&\qquad \frac{d}{dt}\int_\Sigma \bar S(\sigma)\, g\left(\frac{ S(t,\sigma)}{\bar S(\sigma)}\right)\,d\sigma \nonumber\\
&= \int_\Sigma \left(1 - \frac{\bar S(\sigma)}{S(t,\sigma)}\right)\,\frac{dS(t,\sigma)}{dt}\,d\sigma,\nonumber\\
&= \int_\Sigma \left(1 - \frac{\bar S(\sigma)}{S(t,\sigma)}\right)\,\left(\lambda(\sigma) - \mu S(t,\sigma) - \eta(\sigma)F(t)S(t,\sigma)\right)\,d\sigma,\nonumber\\
&= -\mu \int_\Sigma S(t,\sigma)\left(1 - \frac{\bar S(\sigma)}{S(t,\sigma)}\right)^2\,d\sigma +  \int_\Sigma \bar v(\sigma)\left(1 - \frac{\bar S(\sigma)}{S(t,\sigma)}\right)\,d\sigma\nonumber\\
&\qquad  - F(t) \int_\Sigma \eta(\sigma)\left(S(t,\sigma) - \bar S(\sigma)\right)\,d\sigma,\label{eq:der3}
\end{align}
where in the final line we have used the identity 
\begin{equation}
\lambda(\sigma) = \mu \bar S(\sigma) + \eta(\sigma)\bar F \bar S(\sigma) = \mu \bar S(\sigma) + \bar v(\sigma).\nonumber
\end{equation}
Similarly, differentiating the second term and substituting in the properties of $\kappa(\tau)$ (eq.~\eqref{eq:kappacons}) gives
\begin{align}
&\qquad \frac{d}{dt}\int_\Sigma\int_{t-\bar\tau}^{t} \kappa(t-s) \bar v(\sigma)\,g\left(\frac{F(s)S(s, \sigma)}{\bar F \bar S(\sigma)}\right)\,ds\,d\sigma\nonumber\\
&= \int_\Sigma\frac{d}{dt}\int_{t-\bar\tau}^{t} \kappa(t-s) \bar v(\sigma)\,g\left(\frac{F(s)S(s, \sigma)}{\bar F \bar S(\sigma)}\right)\,ds\,d\sigma\nonumber\\
&= \int_\Sigma\left[\kappa(0)\bar v(\sigma)\,g\left(\frac{F(t)S(t,\sigma)}{\bar F \bar S(\sigma)}\right) - \kappa(\bar\tau)\bar v(\sigma)\,g\left(\frac{F(t-\bar\tau)S(t - \bar\tau,\sigma)}{\bar F \bar S(\sigma)}\right)\right.\nonumber\\
&\qquad\qquad + \left. \int_{t-\bar\tau}^t \frac{d\kappa(t-s)}{dt}\bar v(\sigma)g\left(\frac{F(s)S(s, \sigma)}{\bar F \bar S(\sigma)}\right)ds\right]\,d\sigma,\nonumber\\
&= \int_\Sigma \bar v(\sigma) g\left(\frac{F(t)S(t,\sigma)}{\bar F \bar S(\sigma)}\right)\,d\sigma - \bar \eta \int_\Sigma\int_{t-\bar\tau}^t A(t-s)\bar v(\sigma) g\left(\frac{F(s)S(s, \sigma)}{\bar F \bar S(\sigma)}\right)\,ds\,d\sigma.\nonumber
\end{align}
Next, we substitute in the definition of $g(x)$ and collect like terms to get
\begin{align}
&\frac{d}{dt}\int_\Sigma\int_{t-\bar\tau}^{t} \kappa(t-s) \bar v(\sigma)\,g\left(\frac{F(s)S(s, \sigma)}{\bar F \bar S(\sigma)}\right)\,ds\,d\sigma\nonumber\\
&\qquad = F(t)\int_\Sigma \eta(\sigma)\left(S(t,\sigma) - \bar S(\sigma)\right)\,d\sigma\nonumber\\
&\qquad \qquad - \int_\Sigma \bar v(\sigma)\left[\log\left(\frac{F(t)S(t,\sigma)}{\bar F\bar S(\sigma)}\right) - \bar\eta\int_{t-\bar\tau}^t A(t-s)\log\left(\frac{F(s)S(s,\sigma)}{\bar F\bar S(\sigma)}\right)\,ds\right]d\sigma.\nonumber
\end{align}
Note that in simplifying this expression we have also used the identity~\eqref{eq:identity}.

Next, given the identity~\eqref{eq:identity}, we can bound the final term in this expression from above using Jensen's inequality:
\begin{align}
\bar\eta\int_{t-\bar\tau}^t A(t-s)\log\left(\frac{F(s)S(s,\sigma)}{\bar F\bar S(\sigma)}\right)\,ds &\leq \log\left[\frac{\bar\eta}{\bar F\bar S(\sigma)}\int_{t-\bar\tau}^t A(t-s) F(s)S(s,\sigma)\,ds\right]\nonumber\\
&\leq \log\left(\frac{\bar\eta G(t,\sigma)}{\bar F \bar S(\sigma)}\right)\nonumber
\end{align}
where 
\begin{equation}
G(t,\sigma) = \int_{t-\bar\tau}^t A(t-s) F(s)S(s,\sigma)\,ds.
\end{equation}
Substituting this result back in we then have
\begin{align}
&\frac{d}{dt}\int_\Sigma\int_{t-\bar\tau}^{t} \kappa(t-s) \bar v(\sigma)\,g\left(\frac{F(s)S(s, \sigma)}{\bar F \bar S(\sigma)}\right)\,ds\,d\sigma\nonumber\\
&\qquad \leq F(t)\int_\Sigma \eta(\sigma)\left(S(t,\sigma) - \bar S(\sigma)\right)\,d\sigma + \int_\Sigma \bar v(\sigma) \log\left(\frac{\bar\eta G(t,\sigma)}{F(t)S(t,\sigma)}\right)\,d\sigma.\label{eq:der4}
\end{align}
Combining~\eqref{eq:der3} and~\eqref{eq:der4} then yields
\begin{align}
\frac{dW(\mc S_t, \mc F_t)}{dt} &\leq -\mu \int_\Sigma S(t,\sigma)\left(1 - \frac{\bar S(\sigma)}{S(t,\sigma)}\right)^2\,d\sigma\nonumber\\ 
&\qquad + \int_\Sigma \bar v(\sigma)\left[1 - \frac{\bar S(\sigma)}{S(t,\sigma)} + \log\left(\frac{\bar\eta G(t,\sigma)}{F(t)S(t,\sigma)}\right)\right]\,d\sigma.\nonumber
\end{align}
In order to demonstrate that the expression on the right-hand side is indeed non-positive, we first add and subtract the expression
\begin{equation}
\int_\Sigma \bar v(\sigma)\log\left(\frac{\bar S(\sigma)}{S(t,\sigma)}\right)\,d\sigma\nonumber
\end{equation}
to get
\begin{align}
\frac{dW(\mc S_t, \mc F_t)}{dt} &\leq -\mu \int_\Sigma S(t,\sigma)\left(1 - \frac{\bar S(\sigma)}{S(t,\sigma)}\right)^2\,d\sigma\nonumber\\ 
&\qquad - \int_\Sigma \bar v(\sigma)\left[g\left(\frac{\bar S(\sigma)}{S(t,\sigma)}\right) + \log\left(\frac{\bar\eta G(t,\sigma)}{F(t)\bar S(\sigma)}\right)\right]\,d\sigma.
\end{align}
Secondly, we add a zero term:
\begin{align}
&\qquad \int_\Sigma \bar v(\sigma)\left[1 - \frac{\bar\eta G(t,\sigma)}{F(t)\bar S(\sigma)}\right]\,d\sigma\nonumber\\
&= \bar F\left[\int_\Sigma \eta(\sigma)\bar S(\sigma)\,d\sigma - \frac{\bar\eta}{F(t)}\int_\Sigma \eta(\sigma) G(t,\sigma)\,d\sigma\right],\nonumber\\
&= \bar F\left[\bar\eta - \bar\eta\right],\nonumber\\
&= 0\nonumber
\end{align}
to the right-hand side to finally obtain
\begin{align}
\frac{dW(\mc S_t, \mc F_t)}{dt} &\leq -\mu \int_\Sigma S(t,\sigma)\left(1 - \frac{\bar S(\sigma)}{S(t,\sigma)}\right)^2\,d\sigma\nonumber\\
&\qquad \qquad  - \int_\Sigma \bar v(\sigma)\left[g\left(\frac{\bar S(\sigma)}{S(t,\sigma)}\right) + g\left(\frac{\bar\eta G(t,\sigma)}{F(t)\bar S(t,\sigma)}\right)\right]\,d\sigma.\label{eq:dW}
\end{align}
Since $g(x) \geq 0$ we have that $dW/dt \leq 0$. Moreover, from~\eqref{eq:dW} we see that the largest invariant subset in $\widehat\Omega$ for which $\dot{W} = 0$ consists only of the endemic equilibrium point $\bar{P}$. Therefore, since the orbit is eventually precompact, by LaSalle's extension to Lyapunov's asymptotic stability theorem~\citep[Theorem~5.17]{hsmith_textbook}, the endemic equilibrium point $\bar{P}$ is globally asymptotically stable.


\end{proof}

\section{Acknowledgements}

The authors would like to gratefully acknowledge Prof. Johannes M{\" u}ller for providing several key suggestions during the preparation of this manuscript.




 \bibliographystyle{spbasic} 
\bibliography{References}





\end{document}